\documentclass[12pt]{article}
\usepackage[utf8]{inputenc}
\usepackage{amsmath,amssymb,amsfonts,amsthm}
\usepackage{tikz}
\usepackage{fullpage}

\title{Cops and Robber game in higher-dimensional manifolds with spherical and
Euclidean metric}
\author{Vesna Ir\v si\v c$^{b,c}$\thanks{The work was done while the author held a postdoctoral position at the Simon Fraser University, supported by NSERC Grant R611450. Supported also by the Slovenian Research Agency (research core funding P1-0297 and projects J1-2452, J1-1693, N1-0095, N1-0218).} 
\and 
Bojan Mohar$^{a,c}$\thanks{Supported in part by the NSERC Discovery Grant R611450 (Canada) and by the Research Project J1-2452 of ARRS (Slovenia). On leave from IMFM, Jadranska 19, 1000 Ljubljana.} 
\and 
Alexandra Wesolek$^{a}$\thanks{Supported by the Vanier Canada Graduate Scholarships program.}}
\date{\today}

\newtheorem{theorem}{Theorem}
\newtheorem{proposition}{Proposition}
\newtheorem{lemma}{Lemma}

\newtheorem{remark}{Remark}

\newcommand{\eps}{\varepsilon}
\DeclareMathOperator{\pr}{pr}
\DeclareMathOperator{\proj}{proj}
\DeclareMathOperator{\diam}{diam}
\newcommand{\un}{\underline}
\newcommand\RR{\ensuremath{\mathbb{R}}}
\newcommand\NN{\ensuremath{\mathbb{N}}}

\newcommand\Val{\mathop{{\hbox{\sc Val}}}}
\newcommand\ValC{\underline{\Val}_{\,C}}
\newcommand\ValR{\overline{\Val}_{R}}

\begin{document}

\maketitle

\begin{center}
$^a$ Department of Mathematics, Simon Fraser University, Burnaby, BC, Canada\\
\medskip

$^b$ Faculty of Mathematics and Physics, University of Ljubljana, Slovenia\\
\medskip

$^c$ Institute of Mathematics, Physics and Mechanics, Ljubljana, Slovenia\\
\medskip
\end{center}

\begin{abstract}
The recently introduced variation of the game of cops and robber is played on geodesic spaces. In this paper we establish some general strategies for the players, in particular the generalized radial strategy and the covering space strategy. Those strategies are then applied to the game on the $n$-dimensional ball, the sphere, and the torus.
\end{abstract}


\noindent
{\bf AMS Subj.\ Class.\ (2020)}: 54E35, 91A23, 91A50

\section{Introduction}
In this paper we study the recently introduced game of Cops and Robber played on geodesic spaces,  see \cite{Mo21}. In this game, the cops are trying to come as close to the robber as possible, whereas the robber is trying to escape. The moves of the players are discrete, and the maximum length of each move is the same for all the players but different in each step. The precise definition is given in the next section. We explore several fundamental winning strategies for the cops and escape strategies for the robber in various typical geodesic spaces like $n$-dimensional balls, spheres, and other manifolds. We investigate the distinction between the cops winning the game (getting arbitrarily close to the robber), and the cops actually catching the robber (one of them having the exact same position as the robber). The min-max theorem holds for this variation of the game \cite{Mo21}, and it is known that three cops suffice to win the game on orientable surfaces of genus 0 or 1 \cite{Mo22}. While this variation of the game is different from the version discussed by Bollob\'as, Leader and Walters \cite{BoLeWa12}, it is more similar to the discrete Cops and Robber game played on graphs.

Historically, pursuit-evasion games have  mostly been studied in the setup of differential games with one  \cite{Is65,Ji15,Le94,Pa70,Pe93} or more pursuers \cite{HaBr74,Ch76,Ps76,LePa85,FeIbAlSa20}. The Lion and Man problem that was proposed by Richard Rado in the late 1930s and presented in Littlewood's Miscellany \cite{Li53,Li86} is a version of the game with one pursuer (the Lion) and one evader (the Man), moving within a circular arena (unit disk in the plane), both running with equal maximum speed. The cops and robber game played on graphs with one cop and one robber was independently introduced by Nowakowski and Winkler \cite{NoWi83} and Quilliot \cite{Qui78}, while the version with more than one cop was introduced by  Aigner and Fromme \cite{AiFr84}. For each graph $G$ and a positive integer $k$, the \emph{Cops and Robber game} on $G$, involves two players. The first player controls $k$ \emph{cops} placed at the vertices of the graph, and the second player controls the \emph{robber}, who is also positioned at some vertex. While the players alternately move to adjacent vertices (or stay at their current position), the cops want to catch the robber and the robber wants to prevent this ever to happen. The main question is how many cops are needed on the given graph $G$ in order that they can guarantee the capture. The minimum such number of cops is the \emph{cop number} $c(G)$ of the graph (see Bonato and Nowakowski \cite{BoNo11} for an enlightening exposition about the game played on graphs).
Our version of the game of cops and robber is played on an arbitrary geodesic space (a compact, path-connected space endowed with intrinsic metric), modifying the Lion and Man game to the rules of the game played on graphs.

In the rest of the paper, we first state needed definitions and some preliminary results. Then the cops and robber game played on an $n$-dimensional ball is studied, which leads to the generalized radial strategy, which is applied to the $n$-dimensional sphere. Next, the covering method is presented and applied to the $n$-dimensional torus. We conclude the paper by presenting a space on which no finite number of cops can catch the robber.

\section{Preliminaries}

\subsection{Geodesic spaces}
Let $(X,d)$ be a metric space. For $x,y\in X$, an \emph{$(x,y)$-path} is a continuous map $\gamma: I\to X$ where $I=[0,1]$ is the unit interval on $\RR$ and $\gamma(0)=x$ and $\gamma(1)=y$. The space is \emph{path-connected} if for any $x,y\in X$, there exists an $(x,y)$-path connecting them.

One can define the \emph{length} $\ell(\gamma)$ of the path $\gamma$ by taking the supremum over all finite sequences $0=t_0<t_1<t_2< \cdots < t_n=1$ of the values $\sum_{i=1}^n d(\gamma(t_{i-1}),\gamma(t_i))$. Note that $\ell(\gamma)$ may be infinite; if it is finite, we say that $\gamma$ is \emph{rectifiable}. Clearly, the length of any $(x,y)$-path is at least $d(x,y)$.  The metric space $X$ is a \emph{geodesic space} if for every $x,y\in X$ there is an $(x,y)$-path whose length is equal to $d(x,y)$.

An $(x,y)$-path $\gamma$ is \emph{isometric} if $\ell(\gamma) = d(x,y)$. Observe that for $0\le t < t' \le 1$ the subpath $\gamma|_{[t,t']}$ is also isometric. Therefore the set $\gamma(I) = \{\gamma(t)\mid t\in I\}$ is an isometric subset of $X$. With a slight abuse of terminology, we say that the image $\gamma(I)\subset X$ is an \emph{isometric path} in $X$.

A path $\gamma$ is a \emph{geodesic} if it is locally isometric, i.e., for every $t\in [0,1]$ there is an $\varepsilon>0$ such that the subpath $\gamma|_J$ on the interval $J = [t-\varepsilon,t+\varepsilon]\cap[0,1]$ is isometric. A path with $\gamma(0)=\gamma(1)$ is called a \emph{loop} (or a \emph{closed path}). When we say that a loop is a geodesic, we mean it is geodesic as a path and it is also locally isometric around its base point, i.e. $\gamma|_{[1-\varepsilon,1]\cup[0,\varepsilon]}$ is isometric for some $\varepsilon>0$.

One can consider any path-connected compact metric space $X$ and then define the shortest-path distance. For $x,y\in X$, the \emph{shortest-path distance} from $x$ to $y$ is defined as the infimum of the lengths of all $(x,y)$-paths in $X$. If any two points in $X$ are joined by a path of finite length, then the shortest path distance gives the same topology on $X$. Compactness implies that any sequence of $(x,y)$-paths contains a point-wise convergent subsequence, and that the limit points determine an $(x,y)$-path. This implies that there is a path whose length is equal to the infimum of all path lengths. Hence, for this metric, which is also known as the \emph{intrinsic metric}, $X$ is a geodesic space. If the intrinsic metric induces a norm, we will denote the norm by $ \Vert  \cdot  \Vert $.

We refer to \cite{BuBuIv01} and \cite{BuSh04} for further details on metric geometry, and on geodesic spaces in particular.
From now on we will assume that $X$ is a compact path-connected space, endowed with intrinsic metric; in other words, $X$ is a compact geodesic space.

\subsection{The game of Cops and Robber on geodesic spaces}

Detailed definition and basic properties of the game can be found in \cite{Mo21, Mo22}, here we only present the necessary definitions.

Let $X$ be a compact, path-connected metric space endowed with intrinsic metric $d$, and let $k\ge1$ be an integer. The space $X$ is also referred to as the \emph{game space}. The \emph{Game of Cops and Robber} on $X$ with $k$ cops is defined as follows. The first player, who controls the robber, selects the \emph{initial positions} for the robber and for each of the $k$ cops. Formally, this is a pair $(r^0,c^0)\in X^{k+1}$, where $r^0\in X$ is the robber's position and $c^0 = (c_1^0,\dots,c_k^0)\in X^k$ are the positions of the cops. The same player selects his \emph{agility function}, which is a map $\tau: \NN\to\RR_+$ which defines the lengths of the steps of the game.
The agility function must allow for the total duration of the game to be infinite, which means that $\sum_{n\ge1} \tau(n) = \infty$.

After the initial position and the agility function are chosen, the game proceeds as a discrete game in consecutive steps. Having made $n-1$ steps $(n\ge1)$, the players have their positions $(r^{n-1},c_1^{n-1},\dots,c_k^{n-1})\in X^{k+1}$. The $n$th step will have its duration determined by the agility: the move will last for time $\tau(n)$, and each player can move with unit speed up to a distance at most $\tau(n)$ from his current position.
So, the robber moves to a point $r^n\in X$ at distance at most $\tau(n)$ from its current position, i.e. $d(r^{n-1}, r^n)\le \tau(n)$. The destination $r^n$ is revealed to the cops. Then each cop $C_i$ ($i\in[k]$) selects his new position $c_i^n$ at distance at most $\tau(n)$ from its current position, i.e. $d(c_i^{n-1}, c_i^n)\le \tau(n)$. The game stops if $c_i^n = r^n$ for some $i\in[k]$. In that case, the \emph{value of the game} is 0 and we say that the cops \emph{have caught} the robber. Otherwise the game proceeds.
If it never stops, the \emph{value of the game} is
\begin{equation}\label{eq:value of game}
v = \inf_{n\ge0} \min_{i\in[k]} d(r^n, c_i^n).
\end{equation}
If the value is 0, we say that the \emph{cops won} the game; otherwise the \emph{robber wins}. Note that the cops can win even if they never catch the robber.

The minimum integer $k$ such that $k$ cops win the game on $X$ is denoted by $c(X)$ and called the \emph{cop number} of $X$. If such a $k$ does not exist, then we set $c(X)=\infty$. Similarly we define the \emph{strong cop number} $c_0(X)$ as the minimum $k$ such that $k$ cops can always catch the robber.

As mentioned in~\cite{Mo22}, with a growing number of cops, the value of the game tends to 0 for every compact $X$. While this does not imply that $c(X)$ is finite, we still believe it to be true.

However, there are compact geodesic spaces where the number of cops needed to catch the robber is unbounded, $c_0(X)=\infty$. An example of such a space can be found in Section~\ref{sec:unboundedc0}, and indicates that $c_0$ might be related to the dimension of the game space.

\subsection{Strategies and $\eps$-approaching game}

A \emph{strategy of the robber} is a function $s \colon X\times X^k\times \RR_+ \to X$, $(r,c,t) \mapsto r'$, such that $d(r,r')\le t$. This strategy tells us to move the robber from $r$ to $r'$ along some geodesic of length $d(r,r')$. A \emph{strategy of cops} is a function $q \colon (r',c,t) \mapsto c'$ such that each cop $C_i$ moves from his current position $c_i$ to a point $c_i'$ at distance at most $t$ from $c_i$.

Using agility $\tau$, initial position $(r^0,c^0)$ and strategies $s,q$ of the robber and the cops, we denote by $v_\tau(s,q)$ the value of the game when it is played using these strategies. Here we implicitly assume the initial position $(r^0,c^0)$ and agility $\tau$ selected by the robber are part of the strategy $s$. The \emph{guaranteed outcome} for the robber (with fixed or arbitrary agility) is 
$$
    \ValR(\tau) = \inf_q \sup_s v_\tau(s,q) \quad \textrm{and} \quad  \ValR = \sup_\tau \ValR(\tau),
$$
where $q$ and $s$ run over all strategies of the cops and the robber (respectively). Similarly, the \emph{guaranteed outcome} for the cops (with fixed or arbitrary agility) is
$$
    \ValC(\tau) = \sup_s \inf_q v_\tau(s,q) \quad \textrm{and} \quad  \ValC = \sup_\tau \ValC(\tau).
$$
If $\ValC = 0$, then we say that \emph{cops win} the game. If $\ValR>0$, then the \emph{robber wins}. In~\cite{Mo21} it was shown that $\ValC = \ValR$. 

Let the agility $\tau$ and a positive integer $N$ be fixed. Considering only $N$ steps of the game, $T=T_N(\tau)=\sum_{i=1}^N \tau(i)$ is the duration of the game during these $N$ steps. We say that these $N$ steps represent an \emph{$\varepsilon$-approaching game for agility $\tau$} if the cops have a strategy $q_\varepsilon$ such that within these $N$ steps, their distance from the robber is at most $\ValC(\tau) + \varepsilon$. It was shown in~\cite{Mo22} that for every agility $\tau$ and every $\varepsilon>0$, there is an $\varepsilon$-approaching game with finitely many steps.

After making $N$ steps of the $\varepsilon$-approaching game for agility $\tau$, using strategy $q_\varepsilon$ of the cops, the players come to a certain position and then they continue playing. Now, the cops can use an $\varepsilon/2$-approaching game for the remaining agility in steps $N+1,N+2,\dots$ using strategy $q_{\varepsilon/2}$. By definition of the $\varepsilon/2$-approaching game, they come within distance $\varepsilon/2$ from the value of the game. Next, they can use $\varepsilon/3$-approaching game for the remaining agility, and so on. If the agility $\tau$ is decreasing, the strategies $q_{\varepsilon}, q_{\varepsilon/2}, q_{\varepsilon/3},\dots$ can be combined into a single strategy that is optimal for the game on $X$ since the cops come arbitrarily close to the value of the game.

Given an $\varepsilon>0$, $c_\varepsilon(X)$ is defined as the minimum number of cops that guarantee to win the $\varepsilon$-approaching game on $X$. With this notation, the following result holds. 

\begin{theorem}[\cite{Mo22}]
\label{thm:eps-game}
If $X$ is a compact geodesic space, then $$c(X) = \sup \{ c_\varepsilon(X) \mid \varepsilon>0 \}.$$
\end{theorem}

\subsection{Guarding}

Let $X$ be the game space and let $A \subseteq X$. The \emph{shadow} of the robber is $\sigma \colon X \to X$ such that $\sigma |_{A} = id |_A$ and $d(x, y) \geq d(\sigma(x), \sigma(y))$. Note that the second condition is also referred to as $\sigma$ being a \emph{$1$-Lipschitz function}. If a cop comes to the point $\sigma(r)$, where $r$ is the robber's position, then we say that the cop \emph{caught the shadow} $\sigma(r)$ of the robber. It is clear that once the cop is in the robber's shadow, he can stay in the shadow forever. Thus, if the robber enters $A$, then the cop will catch the robber. So we say that the cop \emph{guards} $A$.

\begin{lemma}
\label{lem:subset}
Let $B$ and $M$ be geodesic spaces, and let $B \subseteq M$ be such that there exists a $1$-Lipschitz mapping $\sigma \colon M \to B$ with the property $\sigma|_B = id|_B$. Then $c(M) \geq c(B)$ and $c_0(M) \geq c_0(B)$.
\end{lemma}

Remark: Lemma \ref{lem:subset} gives a lower bound for $c(M)$. Further, if $k$ cops have a catch/win strategy on $B$, the strategy can be used for guarding $B$ in $M$.

\begin{proof}
Let $c(M) = k$. We describe a strategy which ensures that $k$ cops can win a game on $B$. Cop's imagine that the game is played on $M$, they make their optimal moves in $M$, and then use $\sigma$ to determine their moves in $B$. Since $\sigma$ is well-defined and $1$-Lipschitz, the mapped movements of the cops are legal, and if $k$ cops can win the game on $M$, their images also win on $B$. The same argument works for catching the robber.
\end{proof}

\subsection{The $n$-dimensional ball, sphere, and torus}

Let $B^n = \{(x_1, \ldots, x_{n-1}, z) = (\underline{x}, z) \in \mathbb{R}^n\, | \;x_1^2 + \cdots + x_{n-1}^2 + z^2 \leq 1\}$ be the $n$-dimensional ball, $n \geq 1$. Let $\pr_2$ be the projection on the $z$-coordinate, i.e., $\pr_2 (\underline{x}, z) = z$, and let $\pr_1$ be the projection on the first $n-1$ coordinates, i.e., $\pr_1 (\underline{x}, z) = \un{x}$.

Let $S^n = \{(x_1, \ldots, x_n, z) = (\underline{x}, z) \in \mathbb{R}^{n+1}\, | \;x_1^2 + \cdots + x_n^2 + z^2 = 1\}$ be the $n$-dimensional sphere, $n \geq 1$. Let $S^n_+$ and $S^n_-$ denote the upper and lower hemisphere, respectively. Let $N = (0, \ldots, 0, 1)$ be the north pole. Note that the projections $\pr_2$ and $\pr_1$ can be used on the sphere as well. (We simply use $\pr_1$ to denote the projection on all but the last coordinate.)

The \emph{great circle} on $S^n$ is the intersection of $S^n$ with a 2-plane that passes through the origin in the Euclidean space $\mathbb{R}^{n+1}$. Observe that the intersection $S^n \cap \{(\un{x}, z)\mid z = z_0\}$ is an $(n-1)$-dimensional sphere of radius $\sqrt{1 - z_0^2}$. If $d(N, (\un{x}, z)) = t \leq \frac{\pi}{2}$, then $ \Vert \pr_1(\un{x}, z) \Vert  = \sin t$ and $|\pr_2(\un{x}, z)| = \cos t$. Points $(\un{y}, w)$ on the geodesic between $N$ and $(\un{x}, z)$ have the property that $\pr_1(\un{y}, w) = \alpha \un{x}$, where $\alpha \in [0, 1]$.

Let $T^n = S^1 \times \cdots \times S^1$ denote the $n$-dimensional flat torus, $n \geq 1$. Equivalently, $T^n$ is the quotient of $\mathbb{R}^n$ under integer shifts in all $n$ directions.

\subsection{The game played on a disk}

While Besicovitch's strategy for the Lion and Man problem on the disk shows that man can escape a lion inside a circular arena~\cite{Li86},
a similar argument can be used to prove that one cop suffices to win our version of the game on a disk (the cop does not catch the robber, but gets arbitrarily close). Note that this result is later generalized (see the proof of Proposition~\ref{prop:radial1}). We first prove the following lemma.

\begin{lemma}
\label{lem:ball-guard}
Let the cops and robber game be played on the unit disk $B^2$. If cop $c$ is on the line between the center $O$ of $B^2$ and the position of the robber $r$, then the cop can guard the disk of radius $d(O, c)$ and center $O$.
\end{lemma}

\begin{proof}
Let $B_c$ be the ball of radius $d(O, c)$. Suppose that the robber moves from $r \notin B_c$ to $r_1 \in B_c$. Let $r'$ be the first point on his trajectory that lies in $B_c$. Then $\triangle O c r'$ is an isosceles triangle, so $\angle r' c O < \frac{\pi}{2}$, thus $\angle r c r' > \frac{\pi}{2}$, thus $|r r'| > |c r'|$. Thus in the cop's move, he first moves from $c$ to $r'$, and then he has enough agility left to catch the robber in $r_1$. 

Suppose the robber moves from $r$ to $r_2 \notin B_c$. Let $c'$ be the point on the line between $O$ and $r_2$ such that $\triangle O c c'$ and $\triangle O r r_2$ are similar (thus $|c c'| < |r r_2|$). Then the cop can end his move on $c'$. So after the robber's move, the cop maintains the position on the line between $O$ and the robber and is guarding the ball with center $O$ and radius $d(O, c')$.
\end{proof}

Next, we prove that one cop can win the game on a disk.

\begin{proposition}
\label{prop:disk}
Let $B^2$ be a 2-dimensional disk. Then $c(B^2) = 1$.
\end{proposition}

\begin{proof}
At the beginning of the game, the robber chooses agility $\tau$ with $\sum_n \tau(n) = \infty$, and an initial position of both players. Let $O$ denote the center of the disk $B^2$, and let $r_n$ and $c_n$ denote positions of the robber and the cop after the $n$-th move, respectively.

The cop's strategy is to first move to the center $O$ of the disk, and in the next steps to move for the same distance as the robber while ending his move on the line connecting the center of the disk and current robber's position while getting as close to robber as possible (see Figure~\ref{fig:onemove}).
By Lemma~\ref{lem:ball-guard}, we may assume that $d(O, r_n)$ is increasing (otherwise the cop gets closer to the robber even sooner), and since $d(O, r_n) \leq 1$, the sequence converges, say $\ell_R = \lim_{n \to \infty} d(O, r_n)$.

\begin{figure}[htb]
    \centering
    \begin{tikzpicture}
    \draw (0,0) -- (6,0);
    \draw (0,0) -- (6,2.078460969);
    \filldraw[black] (0,0) circle (2pt) node[label=left:$O$] (o) {};
    \filldraw[black] (4,0) circle (2pt) node[label=below:$r_n$] (rn) {};
    \filldraw[black] (5, 1.73205) circle (2pt) node[label=above:$r_{n+1}$] (rn1) {};
    \filldraw[black] (2,0) circle (2pt) node[label=below:$c_n$] (cn) {};
    \filldraw[black] (3.57143, 1.23718) circle (2pt) node[label=above:$c_{n+1}$] (cn1) {};
    \draw[thick, ->] (rn) -- (rn1);
    \draw[thick, ->] (cn) -- (cn1);
    \end{tikzpicture}
    \caption{One move in the cop's strategy.}
    \label{fig:onemove}
\end{figure}
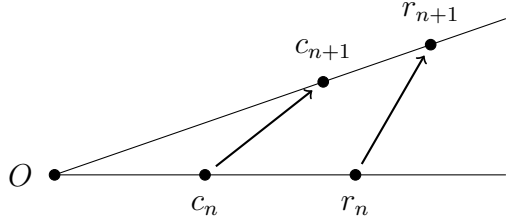

Thus $d(O,c_n)$ is also increasing for all $n \geq N_0$, where $N_0$ is the step of the game in which the cop reaches $O$. Thus $d(O,c_n)$ converges, say $\ell_C = \lim_{n \to \infty} d(O, c_n)$. If $\ell_R = \ell_C$, then the cop can get arbitrarily close to the robber, hence he wins the game.

Suppose that $\ell_R > \ell_C$. Take $\eps = \pi (\ell_R - \ell_C)$.  (Notice that $\eps > 0$ since $0 < \ell_C < \ell_R \leq 1$.) Then there exists $N_1 \in \mathbb{N}$ such that for all $n \geq N_1$ we have $d(O, c_n) \geq \ell_C - \eps$ and $d(O, r_n) \geq \ell_R - \eps$. There also exists $M_1 \in \mathbb{N}$ such that $\sum_{i = N_1}^{M_1} \tau(i) \geq \eps + 2 \pi \ell_R$. In steps $i = N_1, \ldots, M_1$, $ \Vert d(O, r_n) \Vert  \in [\ell_R - \eps,  \ell_R]$, so the robber travels the distance of at least one full circle of circumference $2 \pi \ell_R$ in these steps. But since $ \Vert d(O, c_n) \Vert  \in [\ell_C - \eps, \ell_C]$ and the cop ends each of his moves on the line between $O$ and the robber, he only needs to move for at most $2 \pi \ell_C$ to ensure this. Thus he can move a bit closer to the robber in each step, and in steps $N_1, \ldots, M_1$ he gains at least $2 \pi \ell_R - 2 \pi  \ell_C$. Hence $$d(O, c_{M_1}) \geq d(O, c_{N_1}) + 2 \pi (\ell_R - \ell_C) \geq \ell_C - \eps + 2 \eps = \ell_C + \eps > \ell_C,$$ which is a contradiction.
\end{proof}

\section{Higher-dimensional balls}

The main idea when studying the Cops and Robber Game on the ball originates from Proposition~\ref{prop:disk}. Man's strategy is to move away from the lion, perpendicular to the line connecting his position and the center of the arena. The strategy ensures that the man is never caught, but it turns out that the lion can get arbitrarily close to the man. We use this idea to prove that one cop also suffices to win the game on the $n$-dimensional ball.

\begin{theorem}
\label{thm:n-ball}
If $n \geq 1$, then $c(B^n) = 1$.
\end{theorem}

\begin{proof}
Cop's strategy is to first move to the center $O$ of $B^n$. Suppose that after $k$ steps the cop is positioned in $c^k$ which lies on the line between $O$ and the robber's position is $r^k$. Then the line $O r^k$ and the robber's position $r^{k+1}$ in the next move span a 2-dimensional subspace in $B^n$ which is isometric to $B^2$. So the cop can use the strategy explained in Proposition~\ref{prop:disk} on every step to get arbitrarily close to the robber.
\end{proof}

On the other hand, it turns out that $n$ cops are needed to catch the robber on $B^n$. The following result is analogous to the result for a pursuit-evasion game played on the $n$-dimensional ball which was studied by Croft~\cite{Croft64}. In his variation of the game birds are trying to catch a fly in $B^n$, where all animals have the same maximal speed. We include the proof since our version of the game is defined differently.

\begin{theorem}
\label{thm:ball-catch}
If $n \geq 1$, then $c_0(B^n) = n$.
\end{theorem}

\begin{proof}
We first prove that the robber has a strategy to escape $n-1$ cops from catching him (i.e.\ occupying the same point as the robber). The robber selects his initial position $(\frac12,0,\ldots,0)$ and positions all cops in $(\frac{1}{100},0,\ldots,0)$. He chooses agility $\tau \colon \mathbb{N} \to \mathbb{R}^+$ with $\tau(n) = \frac{1}{n+2}$.

Suppose that the robber is not caught after $k$ steps of the game. Denote players position after $k$th step as $(r^k, c_0^k, \ldots, c_{n-1}^k)$. Let $\Lambda$ be the hyperplane containing points $r^k, c_1^k, \ldots, c_{n-1}^k$. Construct a line $\lambda$, which contains $r^k$ and is perpendicular to $\Lambda$. Let $L$ be the point on $\lambda$ such that $LO \bot \lambda$. The robber moves for time $\tau(k+1)$ along $\lambda$ towards $L$ (and possibly beyond $L$) at maximum speed. If $L \in \Lambda$, the robber may choose his direction arbitrarily. Denote the new position of the robber by $r^{k+1}$. Clearly, no cop can reach the point $r^{k+1}$ in time $\tau(k+1)$. We still need to check that $r^{k+1} \in B^n$.

Denote $|O r^k| = r_k$, $|O r^{k+1}| = r_{k+1}$, and let $\alpha$ be the angle at $r^k$ in the triangle $\triangle O r^k r^{k+1}$. Due to the choice of the direction of robber's move, $\alpha \in [0, \pi/2]$. The law of cosines gives $$r_{k+1}^2 = r_k^2 + \tau(k+1)^2 - 2 r_k \tau(k+1) \cos \alpha \leq r_k^2 + \tau(k+1)^2.$$  Using induction this yields $$r_{k+1}^2 \leq \sum_{i = 1}^{k+1} \tau(i)^2 < \sum_{i = 1}^{\infty} \tau(i)^2 = \frac{\pi^2}{6} - \frac{5}{4} < 1,$$ so $r^{k+1} \in B^n$ and thus the robber is also not caught after $k+1$ moves.

Next, we prove that $n$ cops have a strategy to catch the robber using induction on $n$. For $n=1$, one cop can clearly catch the robber on $B^1 = [-1,1]$. Suppose $n$ cops can catch the robber on $B^n$ and observe the game on $B^{n+1}$. The robber chooses initial positions and agility function $\tau$. First, the cops $c_2, \ldots, c_{n+1}$ all move into the hyperplane with $z = 0$. By induction, these $n$ cops can catch the robber's projection $\pr_1(r)$ in $B^{n+1} \cap \{ z = 0 \}$. Now one cop, say $c_{n+1}$, keeps $\pr_1$ the same as robber's. If he has some agility left, then he also moves closer to the robber in $z$-coordinate. The cops $c_1, \ldots, c_n$ move to the hyperplane $z = \pr_2(r)$. They maintain their position in the hyperplane with the same $z$ coordinate as the robber has. If they have some agility left, they move according to their strategy to catch the robber in the $n$-dimensional ball $B^{n+1} \cap \{ z = \pr_2(r)\}$. 

Let $N_0$ be the step in which both above conditions are fulfilled. The robber either stays close to the $z=\pr_2(r)$ hyperplane or moves away from it. We can partition the set of indices $\{ N_0, N_0 + 1, \ldots \}$ into disjoint sets $M_x$ and $M_z$ in the following way: 
\begin{align*}
    M_z &= \left\{ m \geq N_0\, | \;d(\pr_2(r^{m+1}), \pr_2(r^m)) \geq \frac{\tau(m)}{2} \right\}, \\
    M_x & = \{ N_0, N_0 + 1, \ldots \} \setminus M_z.
\end{align*}
 Note that if $m \in M_x$, then $$d(\pr_1(r^{m+1}), \pr_1(r^m)) = \sqrt{\tau(m)^2 - d(\pr_2(r^{m+1}), \pr_2(r^m))^2} > \frac{\sqrt{3}}{2} \tau(m) > \frac{\tau(m)}{2}.$$ 
 If $\sum_{m \in M_x} d(\pr_1(r^{m+1}), \pr_1(r^m)) $ and $\sum_{m \in M_z} d(\pr_2(r^{m+1}), \pr_2(r^m))$ are finite, then $\sum_{m \geq N_0} \frac{\tau(m)}{2}$ is finite too.
 This implies that $\sum_{m \in \mathbb{N}} \tau(m)  < \infty$ (changed $m \in 1$ to $m$ integer), which is by definition not possible for an agility function $\tau$. Hence at least one of the sums $\sum_{m \in M_x} d(\pr_1(r^{m+1}), \pr_1(r^m))$ and $\sum_{m \in M_z} d(\pr_2(r^{m+1}), \pr_2(r^m))$ must be infinite. If it is the first one (resp.\ the second one), then the cops $c_1, \ldots, c_n$ (resp.\ the cop $c_{n+1}$) can by induction catch the robber. 
\end{proof}

\section{Generalized radial strategy}

The radial strategy used to determine the cop number of the $n$-dimensional ball can be generalized in the following way.

Let $X$ be a geodesic space and $x_0 \in X$ a fixed point. For $y \in X$, a \emph{ray} is a simple $(x_0, y)$-path $R = R(x_0, y)$, $R \colon [0, 1] \to X$, with $R(0) = x_0$, $R(1) = y$, and the property $d(x_0, R(t)) < d(x_0, R(t'))$ if $0 \leq t < t' \leq 1$. We say that $X$ is \emph{starshaped} at $x_0$ if $X = \bigcup_{y \in A} R(x_0, y)$ for some $A \subseteq X$, $R(x_0, y) \cap R(x_0, z) = \{x_0\}$ for every distinct $y, z \in A$, and $d(x_0,y)=d(x_0,z)$ for every $y,z\in A$.

For example, in $B^n$, setting $x_0$ to be the center of the ball, $A$ to be $\partial B^n$, and taking rays as straight lines from $x_0$ to every point in $A$, satisfies the starshaped condition.

Let $X$ be starshaped at $x_0$. Denote $C_d = \{ t \in X \, | \, d(x_0, t) = d \}$. For every $d, d' \in (0, d(x_0, A)]$, we can define a mapping $\varphi_{d',d} \colon C_{d'} \to C_d$. For every $y \in C_{d'}$ there exists exactly one ray containing $y$, say $R(x_0, y')$, $y' \in A$. There exists exactly one point on $R(x_0, y')$ at distance $d$ from $x_0$. Take this point to be $\varphi_{d',d}(y)$.
Notice that $\varphi_{d',d}^{-1} = \varphi_{d,d'}$. Sometimes we will use a simplified notation $\varphi_d(y) = \varphi_{d',d}(y)$ since $d' = d(x_0, y)$ is implicitly expressed with the choice of $y$.

For $d_1, d_2 \geq d$ define 
\begin{align*}
\delta(d_1, d_2, d) & = \inf_{y_1 \in C_{d_1}, y_2 \in C_{d_2}, y_1 \neq y_2} \frac{d(y_1, y_2) - d(\varphi_{d_1, d}(y_1), \varphi_{d_2, d}(y_2))}{d(y_1, y_2)} = \\
& = \inf_{y_1 \in C_{d_1}, y_2 \in C_{d_2}, y_1 \neq y_2} \left( 1 -  \frac{d(\varphi_{d_1, d}(y_1), \varphi_{d_2, d}(y_2))}{d(y_1, y_2)} \right).
\end{align*}

\begin{proposition}
\label{prop:radial1}
If $X$ is starshaped at $x_0$ and $\delta(d_1, d_2, d) > 0$ for every $d_1 \geq d_2 \geq d$, $d_1 > d$, then $c(X) = 1$.
\end{proposition}

\begin{proof}
Let $\tau$, $\sum_n \tau(n) = \infty$, be the agility chosen by the robber. The cop's strategy is first to move to $x_0$ and then try to move closer to the robber, while ending each of his moves on the same ray $g$ as the robber is on.

First, we prove that if the cop is positioned at distance $d$ from $x_0$ while the robber is at distance more than $d$ from $x_0$ and they are both on the same ray $R$, then the cop can guard $C_d$. Notice that since $\delta(d_1, d_2, d) > 0$, we have $d(y_1, y_2) > d(\varphi_{d_1, d}(y_1), \varphi_{d_2, d}(y_2))$ for every $y_1 \in C_{d_1}, y_2 \in C_{d_2}$. Thus if the robber moves such that his distance from $x_0$ remains greater than $d$, then the condition from the statement ensures that the cop can end his move on $C_d$, again on the same ray $R$ as the robber is on. 

If the robber moves closer to $x_0$ than $d$, then the cop first moves to the intersection point between the robber's trajectory and $C_d$ (for which he spends at most the same distance as the robber did), and then follows along the robber's trajectory to catch him.
Thus we may assume that $d(x_0, r_n)$ has a limit. (Otherwise the cop catches the robber or gets closer to him even sooner. Indeed, consider the projection of both players' steps on the interval $[0,d(x_0, A)]$, where the robber's position is $d(x_0, r_n)$ and the cop's position is $d(x_0, c_n)$. If the robber moves at least for distance $d=d(x_0, A)$ steps towards $0$, then the cop caught the robber in $[0,d(x_0, A)]$, and hence in $X$.) Let $\ell_R = \lim_{n \to \infty} d(x_0, r_n)$. The cop's strategy is to first move to $x_0$ (for example along one of the rays $R$), and then use the ``radial'' strategy: ending each of his moves on the same ray $R$ as the robber is on, while trying to get as close to the robber as possible. Since $\delta(d_1, d_2, d) > 0$ for all appropriate $d_1$, $d_2$, $d$, we easily conclude that for every distinct $y_1\in C_{d_1}$, $y_2\in C_{d_2}$, we have $d(y_1, y_2) > d(\varphi_{d_1,d}(y_1), \varphi_{d_2,d}(y_2))$. Thus $d(x_0,c_n)$ is also increasing (after the cop reaches $x_0$) and hence convergent. Let $\ell_C = \lim_{n \to \infty} d(x_0, c_n)$. If $\ell_R = \ell_C$, then the cop can get arbitrarily close to the robber, hence he wins the game.

Suppose that $\ell_R > \ell_C$. Let $\delta_0 = \delta(\frac{\ell_R + \ell_C}{2}, \frac{\ell_R + \ell_C}{2}, \ell_C) >0$. Note that $\delta_0 \leq 1$ by definition. Take $\eps = \frac{\ell_R - \ell_C}{2}$. There exists $N_0 \in \mathbb{N}$ such that for all $n \geq N_0$ we have $d(x_0, c_n) \geq \ell_C - \eps$ and $d(x_0, r_n) \geq \ell_R - \eps$. For $n \geq N_0$, denote $r = r_{n}$, $c = c_{n}$, $r' = r_{n+1}$, $R'$ is the ray on which $r'$ lies, and $\overline{c}$ is the point in $R' \cap C_{d(x_0, c)}$. The cop's strategy is to first move to $\overline{c}$, and then move along $R'$ for the remainder of his agility. 

Observe that, $d(r, r') \geq d(\varphi_{\ell_R - \eps}(r), \varphi_{\ell_R - \eps}(r'))$, and $d(c, \overline{c}) \leq  d(\varphi_{\ell_C}(c), \allowbreak \varphi_{\ell_C}(\overline{c}))$ (here we are using the simplified notation for the function $\varphi$). 
Therefore
$$1 - \frac{d(c,\overline{c})}{d(r,r')} \geq 1 - \frac{d(\overbrace{\varphi_{\ell_C}(c)}^{\varphi_{\ell_C}(y_1)}, \overbrace{\varphi_{\ell_C}(\overline{c}))}^{\varphi_{\ell_C}(y_2)}}{d(\underbrace{\varphi_{\ell_R - \eps}(r)}_{y_1}, \underbrace{\varphi_{\ell_R - \eps}(r'))}_{y_2}} \geq \delta_0.$$
 Thus $d(c, \overline{c}) \leq (1 - \delta_0) d(r,r')$. This means that the cop can move closer to the robber along $g(r')$ for at least $d(r,r') - (1 - \delta_0) d(r,r') = \delta_0 d(r,r')$. Note that this holds for every step $n \geq N_0$. Since $\sum_n \tau(n) = \infty$, there exists $M_0 \in \mathbb{N}$ such that $\sum_{n = N_0}^{M_0} \tau(n) \geq \frac{\ell_R - \ell_C}{\delta_0}$. 

Thus in steps $i = N_0, \ldots, M_0$, the cop is able to move closer to the robber for at least $\sum_{n = N_0}^{M_0}  \delta_0 \tau(n) \geq \allowbreak \delta_0 \cdot \frac{\ell_R - \ell_C}{\delta_0} = \ell_R - \ell_C$. But this means that $d(x_0, c_{M_0}) \geq d(x_0, c_{N_0}) + \ell_R - \ell_C \geq \ell_C - \eps + \ell_R - \ell_C = \ell_R - \eps > \ell_C$, which is a contradiction.
\end{proof}

To illustrate the use of Proposition~\ref{prop:radial1}, we give another proof that $c(B^n) = 1$. 

\begin{lemma}
\label{lem:radial_ball}
The $n$-dimensional ball $B^n$ satisfies the condition that $\delta(d_1, d_2, d) > 0$ for every $d_1 \geq d_2 \geq d, d_1 > d$.
\end{lemma}

\begin{proof}
Let $x_0$ be the center of the ball $B^n$. Taking rays as straight lines between $x_0$ and the points on the boundary of $B^n$ satisfies the starshaped condition.
First, suppose that $d_1 = d_2$. In this case, $d_2 > d$. Take distinct  $y_1, y_2 \in C_{d_1}$. Triangles $\triangle x_0 \varphi_{d_1, d}(y_1) \varphi_{d_1, d}(y_2)$ and $\triangle x_0 y_1 y_2$ are similar. Thus $$\frac{d}{d_1} = \frac{d(\varphi_{d_1, d}(y_1), \varphi_{d_1, d}(y_2))}{d(y_1, y_2)}.$$ So $$\delta(d_1, d_1, d) = \inf_{y_1, y_2 \in C_{d_1},  y_1 \neq y_2} \left( 1 -  \frac{d(\varphi_{d_1, d}(y_1), \varphi_{d_1, d}(y_2))}{d(y_1, y_2)} \right) = 1 - \frac{d}{d_1} > 0.$$

Second, suppose that $d_1 > d_2$ and take $y_1 \in C_{d_1}$, $y_2 \in C_{d_2}$. Since $\triangle x_0 y_2 \varphi_{d_1, d_2}(y_1)$ is an isosceles triangle, the angle $\angle y_1 \varphi_{d_1, d_2}(y_1) y_2 > \frac{\pi}{2}$. Thus $d(y_1, y_2)^2 > (d_1 - d_2)^2 + d(\varphi_{d_1, d_2}(y_1), y_2)^2$. We have \begin{align*}
    \frac{d(\varphi_{d_1, d}(y_1), \varphi_{d_2, d}(y_2))^2}{d(y_1, y_2)^2} & < \frac{d(\varphi_{d_1, d_2}(y_1), y_2)^2}{(d_1 - d_2)^2 + d(\varphi_{d_1, d_2}(y_1), y_2)^2} =\\
    & = 1 - \frac{(d_1 - d_2)^2}{(d_1 - d_2)^2 + d(\varphi_{d_1, d_2}(y_1), y_2)^2} \leq \\
    & \leq 1 - \frac{(d_1 - d_2)^2}{(d_1 - d_2)^2 + 4},
\end{align*}
where we used the bound $d(\varphi_{d_1, d_2}(y_1), y_2)  \leq \diam(B^n) = 2$. Thus $$\delta(d_1, d_2, d) \geq \inf_{y_1 \in C_{d_1}, y_2 \in C_{d_2}} 1 - \sqrt{1 - \frac{(d_1 - d_2)^2}{(d_1 - d_2)^2 + 4}} > 0.$$
\end{proof}

We remark that for $B^n$ the described condition can be simplified to only requiring that for every $d' > d$, $$\delta(d', d) = \inf_{y_1,y_2 \in C_{d'}, y_1 \neq y_2} \left( 1 -  \frac{d(\varphi_{d', d}(y_1), \varphi_{d', d}(y_2))}{d(y_1, y_2)} \right) > 0.$$ This can be rephrased as follows. For $y_1, y_2 \in C_d$, let $\alpha_d(y_1, y_2) := d(y_1, y_2)$. Our condition then requires that $\alpha_d(y_1, y_2)$ is strictly increasing in terms of $d$ for every selection of $y_1, y_2$. 

We also prove that the radial strategy can be used on a hemisphere of an $n$-dimensional sphere. We start by stating the following technical lemma.

\begin{lemma}
\label{lem:technical}
Let $0 < d \leq d_2 \leq d_1 \leq \frac{\pi}{2}$, $d < d_1$. If $0 < \alpha \leq \frac{\pi}{2}$, then $$\cos d_1 \cos d_2 + \sin d_1 \sin d_2 \cos \alpha < \cos^2 d + \sin^2 d \cos \alpha.$$
\end{lemma}

\begin{proof}
Let $\varphi, \psi \colon [0, \frac{\pi}{2}] \to \mathbb{R}$, $\psi(\alpha) = \cos d_1 \cos d_2 + \sin d_1 \sin d_2 \cos \alpha$ and $\varphi(\alpha) = \cos^2 d + \sin^2 d \cos \alpha$. Notice that $\varphi(0) = 1$ and $$\psi(0) = \begin{cases} 1, & d_1 = d_2; \\
\cos d_1 \cos d_2 + \sin d_1 \sin d_2, & d_1 \neq d_2. \end{cases}$$
It is easy to see that $\cos d_1 \cos d_2 + \sin d_1 \sin d_2 < 1$ for every $0 < d_2 < d_1 \leq \frac{\pi}{2}$. On the other hand, $\varphi'(\alpha) = - \sin^2 d \sin \alpha$ and $\psi'(\alpha) = - \sin d_1 \sin d_2 \sin \alpha$. Since $d_1 > d$, $0 > \varphi'(\alpha) > \psi'(\alpha)$. Thus for every $\alpha \in (0, \frac{\pi}{2}]$, $1 > \varphi(\alpha) > \psi(\alpha)$.
\end{proof}
 
\begin{lemma}
\label{lem:radial_sphere}
The $n$-dimensional hemisphere $S^n_+$ satisfies the condition that $\delta(d_1, d_2, d) > 0$ for every $d_1 \geq d_2 \geq d, d_1 > d$.
\end{lemma}

\begin{proof}
The shortest paths between the north pole $N$ and the points on the boundary of $S_+^n$ have all the conditions needed for rays, thus $S_+^n$ is starshaped at $N$.
Let $\tau'$ be the agility function chosen by the robber. Let $\tau$ be a subdivision of $\tau'$ such that the angle at $N$ corresponding to each move is smaller than $\frac{\pi}{2}$. By~\cite[Lemma 5]{Mo21} this only goes in the favour of the robber.

Let $y_1 \in C_{d_1}$, $y_2 \in C_{d_2}$, $y_1 \neq y_2$. If $y_2$ lies on the geodesic between $N$ and $y_1$, then $d(\varphi_{d_1, d}(y_1), \varphi_{d_2, d}(y_2)) = 0$ and $1 -  \frac{d(\varphi_{d_1, d}(y_1), \varphi_{d_2, d}(y_2))}{d(y_1, y_2)} = 1$, so this case can be excluded in the below calculations. Using the spherical law of cosines, we get
\begin{align*}
    \cos \left( d(\varphi_{d_1, d}(y_1), \varphi_{d_2, d}(y_2)) \right) & = \cos^2 d + \sin^2 d \cos \alpha,\\
    \cos \left( d(y_1, y_2) \right) & = \cos d_1 \cos d_2 + \sin d_1 \sin d_2 \cos \alpha,
\end{align*}
where $\alpha$ is the angle between the geodesics from $N$ to $y_1$ and from $N$ to $y_2$. Our choice of the agility function yields that $0 < \alpha \leq \frac{\pi}{2}$. Notice that this allows us to write
$$\delta(d_1, d_2, d) = \inf_{0 < \alpha \leq \pi/2} \left( 1 - \frac{\arccos(\cos^2 d + \sin^2 d \cos \alpha)}{\arccos(\cos d_1 \cos d_2 + \sin d_1 \sin d_2 \cos \alpha)} \right).$$
To determine $\delta(d_1, d_2, d)$, let us consider the continuous function $f \colon (0, \frac{\pi}{2}] \to \mathbb{R}$, $$f(\alpha) = 1 - \frac{\arccos(\cos^2 d + \sin^2 d \cos \alpha)}{\arccos(\cos d_1 \cos d_2 + \sin d_1 \sin d_2 \cos \alpha)}.$$ By Lemma~\ref{lem:technical}, $f(\alpha) > 0$ for every $\alpha \in (0, \frac{\pi}{2}]$.

Since $d_1 > d$, $f(\frac{\pi}{2}) = 1 - \frac{\arccos(\cos^2 d)}{\arccos(\cos d_1 \cos d_2)} > 0$. If $d_1 \neq d_2$, then $$\lim_{\alpha \searrow 0} f(\alpha) = 1 - \frac{\arccos(1)}{\arccos(\cos d_1 \cos d_2 + \sin d_1 \sin d_2)} = 1.$$ If $d_1 = d_2$, then we can write $$f(\alpha) = 1 - \frac{\arccos(1 + \sin^2 d(\cos \alpha - 1))}{\arccos(1 + \sin^2 d_1(\cos \alpha - 1))},$$ from which we can see that  $\lim_{\alpha \searrow 0} f(\alpha)  > 0$.

As $f$ is continuous and its limits at the boundary are bounded away from zero, it follows that $\inf_{\alpha \in (0, \pi/2]} f(\alpha) > 0$ and thus $\delta(d_1, d_2, d) > 0$.
\end{proof}

\section{Higher-dimensional spheres}

It turns out that a constant number of cops are also enough to win the game on $n$-dimensional spheres, while a linear number of cops is needed to catch the robber.

When studying the game on $S^n$ it is useful to consider a mirroring strategy of the cop. Let

$\rho \colon S^n \to S^n_-$, $\rho((x_1, \ldots, x_n, z)) = \begin{cases} (x_1, \ldots, x_n, z), & z \leq 0,\\ (x_1, \ldots, x_n, -z), & z \geq 0. \end{cases}$

\begin{lemma}
\label{lem:guard-equator}
One cop can guard a great circle in $S^n$.
\end{lemma}

Remark: the cop essentially guards a whole hemisphere.

\begin{proof}
Given starting positions $(r_1, c_1)$ of the robber and the cop, respectively, we change coordinates such that $c_1 = \rho(r_1)$ (this can be done with an orthogonal transformation, so distances are preserved). Now the cop's strategy is to move such that $c_n = \rho(r_n)$, thus guarding the great circle $S^n \cap \{ (x_1, \ldots, x_n, 0) \in S^n \}$.
\end{proof}

\begin{theorem}
\label{thm:sphere}
If $n\geq 1$, then $c(S^n) = 2$.
\end{theorem}

\begin{proof}
Let $(r_1, c_1, c_2)$ be the  starting positions of the robber and both cops. First, the coordinate system is fixed in such a way that $c_1 = \rho(r_1)$. Afterwards, the first cop uses the strategy from Lemma~\ref{lem:guard-equator} to guard the lower hemisphere. So the robber is essentially limited to moving around in the upper hemisphere (otherwise he is caught). The second cop first moves to the north pole $N$ and then uses the radial strategy. Proposition~\ref{prop:radial1} and Lemma~\ref{lem:radial_sphere} ensure that he can get arbitrarily close to the robber.
\end{proof}

It follows from~\cite{SaKu00} that exactly $n+1$ cops are needed to catch the robber on $S^n$, i.e.\ $$c_0(S^n) = n+1.$$

Observe that two cops can win on $S^n$ even if each of them is using radial strategy on a complementary hemisphere. This can easily be generalized to spaces that are union of a number of subsets, each satisfying the condition of Proposition~\ref{prop:radial1}.

\section{Covering space method and its application to the flat torus}

Sometimes it is beneficial for the cops to imagine that the game is played on the covering space instead of on the space itself. The following result shows that if $k$ cops can win the game on the covering space, then $k$ cops can also win on the space itself.

For a topological space $X$, the covering space of $X$ is a topological space $C$ together with a continuous surjective map $p \colon C \to X$
such that for every $x \in X$, there exists an open neighborhood $U$ of $x$, such that $p^{-1}(U)$ is a union of disjoint open sets in $C$, each of which is mapped homeomorphically onto $U$ by $p$ (i.e.\ it is a local homeomorphism). If $X$ and $C$ are both geodesic spaces and $p$ is not only a local homeomorphism but also a local isometry, then we say that the covering space $C$ of $X$ \emph{locally preserves distances}.

\begin{lemma}
\label{lem:covering}
If $C$ is the covering space of $X$ that locally preserves distances, then $c(X) \leq c(C)$.
\end{lemma}

\begin{proof}
Let $p \colon C \to X$ be the covering map and let $c(C) = k$. While the game in $X$ is played with $k$ cops, the cops imagine the game is simultaneously played on $C$. If a cop moves to $c$ in $C$, the same cop moves to $p(c)$ in $X$ (which is possible due to the properties of the covering map). Thus since at least one cop can get arbitrarily close to the robber in $C$, the image of this cop in $X$ also gets arbitrarily close to the robber.
\end{proof}

In the following we consider the cops and robber game played on a flat torus, to demonstrate the usefulness of the covering space method. Recall that the covering space of $T^n$ is $\mathbb{R}^n$.

\begin{lemma}
\label{lem:Rn2cops}
Let $\eps >0$. If the robber's position is on the line between two cops in $\mathbb{R}^n$, then the two cops can come to distance $\eps$ from the robber in $\mathbb{R}^n$. 
\end{lemma}

\begin{proof}
Without loss of generality we may assume that $r^0 = (0,\ldots,0)$, $c_1^0 = (0,\ldots,0,a)$, $c_2^0 = (0,\ldots,0,-b)$, $a \geq b > 0$. Take $\eps > 0$. We want to prove that the cops have a strategy to reduce $d(r^n, \{c_1^n, c_2^n\})$ to at most $\eps$. Let $t = \frac{1}{2 \eps} (a^2 - \eps^2)$, $\alpha = \arctan(\frac{a}{t})$ and $\beta = \arctan(\frac{b}{t})$. Note that the right-angled triangle with one of the angles $\alpha$ and legs of lengths $a$ and $t$ has the hypotenuse of length $t + \eps$.

The cops partition each move of the robber into its first $n-1$ coordinates $x_1, \ldots, x_{n-1}$, and the last coordinate $x_n$, denoted by $\tau_n^x$ and $\tau_n^z$, respectively. Note that in each step, at least one of these is greater than $\frac{\tau_n}{2}$. The cops move for $\tau_n^x$ in the same direction as the robber, but slightly towards him, the first cop following a line with slope $\pm \alpha$, and the second cop following a line with slope $\pm \beta$ (as long as this step ensures that the $x_n$ coordinates of the first cop, resp. second cop, is larger, resp. smaller, than the $x_n$ coordinate of the robber, otherwise they move on a smaller angle, directly towards the robber). They also move for $\tau_n^z$ in $x_n$-coordinate towards the robber. The first part of the move at angle $\alpha$ ($\beta$, resp.) ensures that the distance between the cops' and the robber's projections on the $x_1, \ldots, x_{n-1}$-hyperplane are always at most $\eps$ apart, while the second part of the move ensures that the distance between the cops' projections on $x_n$-axis is getting smaller. Notice also that such moves ensure that after every move of the cops, the robber is at most at distance $\eps$ from the projections of cops' positions onto the $x_1, \ldots, x_{n-1}$-hyperplane. 
If $\sum_{n=1}^k \tau_n^x >t$ for some $k$, then the latest after $k$ steps the distance between the cops and the robber will be at most $\eps$. Otherwise, since $\sum_n \tau_n = \infty$, it holds $\sum_n (\tau_n - \tau_n^x) =\sum_n \tau_n^z= \infty$, thus the cops will eventually be on the same $x_n$-coordinate as the robber, and the robber will be at least $\eps$-close.
\end{proof}

\begin{remark}
The conclusion of Lemma~\ref{lem:Rn2cops} also holds if the robber is at distance less than $\eps$ from the line between the two cops. The details are left to the reader.
\end{remark}

Using the covering space method we demonstrate that two cops win the game on $T^n$.

\begin{theorem}
\label{thm:Tnflat}
For every $n \geq 1$, $c(T^n) = 2$.
\end{theorem}

\begin{proof}
As $S^1 \subseteq T^n$ and the projection $T^n \to S^1$ is a $1$-Lipschitz mapping, Lemma~\ref{lem:subset} yields $c(T^n) \geq 2$. 

The universal covering space of $T^n$ is the Euclidean space $\mathbb{R}^n$. We will prove that with good initial positions two cops can win the game on $\mathbb{R}^n$, thus by Lemma~\ref{lem:covering}, $c(T^n) \leq 2$ as well. Without loss of generality we may assume that $c_1^0 = c_2^0 = (0,\ldots,0) \in T^n$. The robber moves from $r^0$ to $r^1$ in his first move. There exists a point with rational coordinates that is at most $\tau_0$ away from $r^1$, say $(\frac{p_1}{q_1}, \ldots, \frac{p_n}{q_n})$, where $p_i, q_i \in \mathbb{N}_0$ for $i \in [n]$. Let $Q = q_1  q_2\cdots q_n$. Imagine the cops' initial positions in the Euclidean space are $c_1^0 = (0,\ldots,0) \in \mathbb{R}^n$ and $c_2^0 = (Q \cdot \frac{p_1}{q_1}, \ldots, Q \cdot \frac{p_n}{q_n})$. They first make the step $r^1-(\frac{p_1}{q_1}, \ldots, \frac{p_n}{q_n})$. As a result, $r^1$ now lies on the line between $c_1^1$ and $c_2^1$. Using Lemma~\ref{lem:Rn2cops}, these two cops can now  come $\eps$-close to the robber in the game on $\mathbb{R}^n$. Thus the two cops win the $\eps$-approaching game on $T^n$. This implies that $c(T^n) \leq 2$.
\end{proof}

On the other hand, $n+1$ cops are needed to catch the robber on $T^n$.

\begin{theorem}
\label{thm:T2flat-catch}
If $n\geq 1$, then $c_0(T^n) = n+1$.
\end{theorem}

\begin{proof}
On $T^n$, $n$ cops cannot catch a robber. Let $t$ be the length of the current robber's step. Since the $t$-ball around $r^k$ can be entirely contained in the union of $t$-balls around $c_1^k, \ldots, c_n^k$ if and only if $r^k = c_i^k$ for some $i \in [k]$, the robber has a direction in which he can move to escape being captured in this move.

Next, we prove in full detail that three cops suffice to catch the robber on $T^2$.
We imagine the game is played on the universal cover space of $T^2$, so on $\mathbb{R}^2$. Without loss of generality, $c_1^0 = c_2^0 = c_3^0 = (0,0)$. Let the robber's initial position be $r^0 \in [0,1]^2$ in the covering space $\mathbb{R}^2$. We imagine the cops' initial positions are $c_1^0 = (0, -100)$, $c_2^0 = (200, 100)$ and $c_3^0 = (-200, 100)$. For $i \in [3]$, let $\ell_i$ be the ray starting at $r_0$ that goes through $c_i^0$, and let $\ell_i^\bot$ be the bisector of the line segment between $c_i^0$ and $r^0$. The selected initial conditions in the covering space ensure that the angle $\alpha_1$ of the slope of $l_1$ is in $[\arctan(100),\pi/2]$, the angle $\alpha_2$ of the slope of $l_2$ is in $[\arctan(\frac{99}{201}),\arctan(\frac{1}{2})]$ and the angle $\alpha_3$ of the slope of $l_3$ is in $[\pi-\arctan(\frac{1}{2}) ,\pi-\arctan(\frac{99}{201})]$. The angles of the lines $l_i^T$
are at $\alpha_i+\pi/2$ or $\alpha_i-\pi/2$.

A short calculation shows that the angles between the bisectors $\ell_i^\bot$ ($i \in [3]$) are between $\frac{\pi}{3} - \frac{\pi}{18} = \frac{5 \pi}{18}$ and $\frac{\pi}{3} + \frac{\pi}{18} = \frac{7 \pi}{18}$ and the angles between the rays $\ell_i$ ($i \in [3]$) are between $\frac{10\pi}{18}$ and $\frac{14 \pi}{18}$.

The cops' strategy is to move in the same direction as the robber if the robber is moving away from the bisector between them, or to move to the reflection along the bisector of the robber's new position. This clearly maintains angles between lines $\ell_i^\bot$, $i \in [3]$, and angles between lines $\ell_i$, $i \in [3]$. Let $D_n = \sum_{i=1}^3 d(r^n, c_i^n)$. The above strategy ensures that $D_n$ is decreasing. Indeed, when the robber moves from $r^n$ to $r^{n+1}$, $d(r^n, r^{n+1}) = \tau_n$, he moves closer to either one or two bisectors between him and a cop. If he moves closer to only one of the bisectors (say closer to the cop $c_j$), then the angle between the move $r^n \to r^{n+1}$ and $\ell_j$ is at most $\frac{7 \pi}{18}$, thus $D_n - D_{n+1} \geq 2 \tau_n \cos(\frac{7 \pi}{18})$.
If the robber moves closer to two of the bisectors (say closer to the cops $c_j$ and $c_k$), then at least one of the angles between the move $r^n \to r^{n+1}$ and $\ell_j$ or $\ell_k$ is smaller than one half of the angle between $\ell_j$ and $\ell_k$, which is smaller than $\frac{7 \pi}{18}$. Thus again, $D_n - D_{n+1} \geq 2 \tau_n \cos(\frac{7 \pi}{18})$.

Since $\sum_n \tau_n = \infty$, there exists $N_0 \in \mathbb{N}$ such that $\sum_{n=1}^{N_0} \tau_n \geq \frac{D_0+1}{2 \cos(\frac{7 \pi}{18})}$. Thus $D_{N_0} = \sum_{i=1}^3 d(r^{N_0}, c_i^{N_0}) \leq D_0 - \sum_{n=1}^{N_0} 2 \tau_n \cos(\frac{7 \pi}{18})  \leq D_0 - 2 \cos(\frac{7 \pi}{18}) \frac{D_0+1}{2 \cos(\frac{7 \pi}{18})} = -1$, which is a contradiction. Hence the cops are able to catch the robber. This concludes the proof for $n=2$.

In higher dimension ($n \geq 3$) the strategy is to choose the initial positions of the cops (in the covering space) such that the robber is caught in a simplex. Then, as in the previous argument, with every step of the robber, there is a cop which gets at least a  positive fraction of the step closer to the robber. We skip the details, but show that there exists such a fixed positive fraction if the robber was caught in a regular $n$-simplex in $\mathbb{R}^{n+1}$.

Let $n+1$ points $S = \{ e_1 = (1,0,\ldots,0), e_2 = (0,1,0,\ldots,0), \ldots, e_{n+1}=(0,\ldots, 0,1) \}$ span the regular simplex $\Sigma$ in $\mathbb{R}^{n+1}$. A point $x = (x_1, \ldots, x_{n+1})$ lies in $\Sigma$ if and only if $x_1 + \cdots + x_{n+1} = 1$ and $x_i \geq 0$ for $i \in [n+1]$. Let the robber's position after step $r$ be $x = (a_1, \ldots, a_{n+1})$ in the interior of $\Sigma$, let $c_i$ be the reflection of $x$ over the hyperplane $\Pi_i$ spanned by $S \setminus \{e_i\}$. The normal of $\Pi_i$ is $n_i = (1,\ldots,1,0,1,\ldots,1)$, where $0$ is on the $i$th coordinate, and $e_{i-1}$ (index modulo $n+1$) lies on $\Pi_i$. Then $c_i = x + 2 \frac{(e_{i-1}-x) \cdot n_i}{n} n_i$. Recall that projection of $a$ onto $b$ is $\proj_b a = \frac{a \cdot b}{b \cdot b} b$. Thus $c_i = x + \frac{2 a_i}{n} n_i$ and $c_i - x = \frac{2 a_i}{n} n_i$.  Notice that  $\Vert c_i - x \Vert^2 = \frac{4 a_i^2}{n}$.

Let the robber's position after step $r+1$ be $x' = (b_1, \ldots, b_{n+1}) \in \Sigma$ and let $\Vert x'-x \Vert = \tau$. Since $\Vert x'-x \Vert^2 = \sum_{j=1}^{n=1} (b_j - a_j)^2 = \tau^2$, at least one of the terms is greater than $\frac{\tau^2}{n+1}$. Without loss of generality, $(b_1 - a_1)^2 \geq \frac{\tau^2}{n+1}$. 

If $b_1 - a_1 < 0$, then $a_1 - b_1 \geq \frac{\tau}{\sqrt{n+1}}$. Thus $(x'-x) \cdot (c_1-x) = \frac{2 a_1}{n} (a_1 - b_1) \geq \frac{2 a_1}{n} \frac{\tau}{\sqrt{n+1}}$. Therefore $\proj_{c_1 - x} (x'-x) \geq \frac{\tau}{\sqrt{n (n+1)}} \frac{c_1-x}{\Vert c_1-x \Vert}$.

If $b_1 - a_1 \geq 0$, then $b_1 - a_1 \geq \frac{\tau}{\sqrt{n+1}}$. Since $0 = \sum_{j=1}^{n=1} (b_j - a_j) \geq \frac{\tau}{\sqrt{n+1}} + \sum_{j=2}^{n=1} (b_j - a_j)$, at least one of the terms is smaller than $\frac{- \tau}{n \sqrt{n+1}}$. Without loss of generality, $b_2 - a_2 \leq \frac{- \tau}{n \sqrt{n+1}}$, i.e., $a_2 - b_2 \geq \frac{\tau}{n \sqrt{n+1}}$. Thus $(x'-x) \cdot (c_2-x) \geq \frac{2 a_2}{n} \frac{\tau}{n \sqrt{n+1}}$ and $\proj_{c_2-x} (x'-x) \geq \frac{\tau}{n \sqrt{n(n+1)}} \frac{c_2-x}{\Vert c_2-x \Vert}$.

Thus it is always true that at least one of $\proj_{c_i-x} (x'-x)$ is directed towards $c_i$ and has norm greater or equal $\frac{\tau}{n \sqrt{n(n+1)}}$.

To prove that $n+1$ cops suffice to catch the robber on $T^n$, we use change of coordinates and the argument explained above.
\end{proof}

\section{A space on which no finite number of cops can catch the robber}
\label{sec:unboundedc0}

The sequence space $\ell^2$ is the space of  square-summable sequences, $\ell^2 = \{ x = (x_1, x_2, \ldots)\, | \;\allowbreak \sum_{i=1}^\infty x_i^2 < \infty \}$. The distance between $x, y \in \ell^2$ is $d(x,y) = \sqrt{\sum_{i=1}^\infty |x_i - y_i|^2}$. 

Let $D \subset \ell^2$ be the subspace $\{ x \in \ell^2 \, | \;  \Vert x \Vert  \leq 1, \forall i \geq 1\colon |x_i| \leq \frac{1}{i} \}$. Since $D$ is a metric space, compactness is equivalent to sequential compactness. Thus it suffices to prove that every sequence in $D$ has a convergent subsequence whose limit belongs to $D$. Let $A = \{a^{(1)}, a^{(2)}, \ldots\}$ be a sequence in $D$. Let us consider the first term of every element in the sequence: $a_1^{(1)}, a_1^{(2)}, \ldots$. For every $j \geq 1$, $a_1^{(j)} \in [-1, 1]$, so there exists a subsequence $a^{(1')}, a^{(2')}, \ldots$ such that the first terms converge to a number in $[-1, 1]$. Within the sequence $a^{(1')}, a^{(2')}, \ldots$ we can find a subsequence such that the second terms converge to a number in $[-\frac12, \frac12]$. Repeating this for every $i \geq 1$ we obtain a subsequence $B = \{ b^{(1)}, b^{(2)}, \ldots \}$ of $A$ with the property that for every $i \geq 1$ the $i$th terms converge to a number in $[-\frac{1}{i}, \frac{1}{i}]$. Thus $A$ has a point-wise convergent subsequence $B$, let $c \in \ell^2$ be its limit. Since for every $j \geq 1$, $ \Vert a^{(j)} \Vert  \leq 1$, we have $ \Vert c \Vert  \leq 1$. And since for every $i \geq 1$, $c_i \in [-\frac{1}{i}, \frac{1}{i}]$, the limiting sequence $c$ is in $D$.

\begin{proposition}
\label{prop:c0-unbounded}
If $D  = \{ x \in \ell^2\, | \; \Vert x \Vert  \leq 1, \forall i \geq 1\colon |x_i| \leq \frac{1}{i} \}$, then $c_0(D) = \infty$.
\end{proposition}

\begin{proof}
For $k \geq 1$ let $D^k = \{ x = (x_1, \ldots, x_k, 0, 0, \ldots) \in \ell^2\, | \;  \Vert x \Vert  \leq 1, \forall i \in [k] \colon |x_i| \leq \frac{1}{i}\}$. Clearly, $D^k \subseteq D$. $D^k$ is compact by the same argument we used to show that $D$ is compact.

Let $\alpha \colon D \to D^k$, $\alpha((x_1, x_2, \ldots)) = (x_1, \ldots, x_k, 0, 0, \ldots)$. Notice that $\alpha|_{D^k} = id|_{D^k}$. For $x, y \in D$ we have $$d(\alpha(x), \alpha(y)) = \sqrt{\sum_{i=1}^k |x_i - y_i|^2} \leq \sqrt{\sum_{i=1}^\infty |x_i - y_i|^2} = d(x,y),$$ thus $\alpha$ is a 1-Lipschitz mapping. So Lemma~\ref{lem:subset} gives $c_0(D) \geq c_0(D^k)$.

Let $M^k$ denote a $k$-dimensional ball of radius $\frac{1}{k}$ (i.e.\ $M^k = \{ x \in \mathbb{R}^k\, | \; \Vert x \Vert  \leq \frac{1}{k}\}$, so in particular $|x_i| \leq \frac{1}{k}$). The same argument as in the proof of Theorem~\ref{thm:ball-catch} gives that $c_0(M^k) = k$. 

We can identify $M^k$ with the ball of radius $\frac{1}{k}$ in $D^k$. Observe that the radial projection $\beta \colon D^k \to M^k$ is $1$-Lipschitz, thus $c_0(D) \geq c_0(D^k) \geq c_0(M^k) = k$ (by Lemma~\ref{lem:subset}). It follows that $c_0(D) = \infty$.
\end{proof}

Note that the game space $D$ from Proposition~\ref{prop:c0-unbounded} is starshaped, thus $c(D) = 1$ by Proposition~\ref{prop:radial1}.

\end{document}